\documentclass[11pt]{article}
\usepackage[utf8]{inputenc}  
\usepackage[T1]{fontenc} 
\usepackage{lmodern}
\usepackage{amsmath} 
\usepackage{amssymb}
\usepackage{amsthm} 
\usepackage{mathrsfs}
\usepackage{graphicx}
\usepackage{ccaption}
\usepackage{hyperref}
\usepackage{listings}
\usepackage{enumitem}
\usepackage{tikz}
\usepackage{float}
\usepackage{pgfplots}
\pgfplotsset{compat=newest}
\usetikzlibrary{shapes}
\usetikzlibrary{patterns}
\usepackage{stmaryrd}
\usepackage{authblk}
\usepackage[pagewise]{lineno}
\usepackage{color}
\usepackage{csquotes}
\usepackage{bm}
\usepackage{multicol}  

\usepackage[left=2.5cm,right=2.5cm,top=3cm,bottom=3cm]{geometry}

\theoremstyle{definition}
\newtheorem{defi}{Definition}
\newtheorem{theorem}{Theorem}
\newtheorem{proposition}{Proposition}

\newtheorem{rem}{Remark}

\newcommand{\ub}{\bm{u}}

\newcommand{\vb}{\bm{v}}

\newcommand{\sib}{\bm{\sigma}}
\newcommand{\Xb}{\bm{X}}

\newcommand{\Yb}{\bm{Y}}

\newcommand{\fb}{\bm{f}}
\newcommand{\bb}{\bm{b}}
\newcommand{\gb}{\bm{g}}

\def\top{\text{top}}
\def\bot{\text{bot}}

\usepackage{fancybox}
\usepackage{bmpsize} \usepackage{pstricks-add}
 \usepackage{mathrsfs}\usetikzlibrary{arrows}
\usepackage{amsmath,esint}
\usepackage{amsfonts}\usepackage{amsthm} \usepackage{euscript} \usepackage{stmaryrd}
\usepackage{braket}
\PassOptionsToPackage{hyphens}{url}
 \usepackage{breakurl}\usepackage{enumitem}
 \usepackage{calligra} \usepackage{dsfont}

\MakeInnerQuote{"}
\setitemize[1]{ label=\textbullet}
\setitemize[2]{ label=$-$}

\title{Asymptotic analysis of Rayleigh-Lamb dispersion relations under various boundary conditions}
\author[1,*]{Angèle Niclas}
\author[2]{Titouan Rocabois}
\affil[1]{Université Paris Cité, CNRS, MAP5, F-75006 Paris, France}
\affil[2]{ENS de Lyon, F-69007 Lyon, France}
\affil[*]{Corresponding author: angele.niclas@math.cnrs.fr}
\date{}

\begin{document}
\maketitle

\begin{abstract}
We investigate the asymptotic behaviour of the complex roots of
Rayleigh-Lamb dispersion relations arising in elastic waveguides.
We consider Neumann, Dirichlet, and fluid boundary conditions and derive
explicit asymptotic expansions of the associated wavenumbers as their
modulus tends to infinity. In the Neumann case, we provide a rigorous
justification of asymptotic formulas that have long been used in the
literature without proof, together with higher-order terms and explicit
constants. Similar results are obtained for Dirichlet and fluid boundary
conditions. The analysis relies on a reformulation of the dispersion relations as
zeros of holomorphic functions and on asymptotic properties of the
Lambert \(W\) function. We also show how these asymptotic expansions can
be used to establish modal decompositions and well-posedness results for
elastic waveguide problems. Numerical experiments confirm the accuracy
of the proposed formulas.
\end{abstract}

\section{Introduction}

The main objective of this paper is to provide a rigorous proof of asymptotic formulas that are used, explicitly or implicitly, in the study of elastic plates and Lamb modes. More precisely, we investigate the behaviour of the wavenumbers $k$ associated with guided elastic waves in a plate in the large-wavenumber regime, namely as $|k|\to\infty$.

\subsection{Lamb waves and associated wavenumbers} 

From a physical point of view, when considering wave propagation in an elastic layer, it is natural to look for time-harmonic waves propagating along the waveguide direction $e_x$ \cite{achenbach1,royer1}. In a two-dimensional setting, such waves are typically sought under the form
\begin{equation}
\ub(x,y,t) = (u(y),v(y))\exp\bigl(i(kx-\omega t)\bigr),
\end{equation}
where $\omega$ denotes the frequency and $k$ the wavenumber. Substituting this ansatz into the elastodynamic equations leads to the classical Rayleigh-Lamb dispersion relations, which provide an implicit relation between $\omega$, $k$, the geometry of the layer, and the Lamé coefficients of the material \cite{achenbach1,fraser1}. Once the admissible values of $k$ have been determined, the corresponding transverse profiles $(u(y),v(y))$, called Lamb modes, can be computed, yielding the theoretical expression of the propagating waves in the waveguide \cite{achenbach1,balogun1}. 

From a mathematical point of view, the search for propagating modes can be formulated as a spectral problem \cite{pagneux2,bonnetier3,niclas4}. The associated eigenvalues determine the admissible wavenumbers $k$, while the corresponding eigenmodes provide the transverse displacement profiles $u(y)$ and $v(y)$. Moreover, it has long been conjectured \cite{besserer1,kirrmann1}, and has recently been established in appropriate functional settings \cite{akian1}, that these modes form a complete family. This completeness property makes it possible to represent sufficiently general elastic fields in the waveguide as infinite sums of Lamb modes \cite{pagneux1,bonnetier3}.

\subsection{Theoretical and numerical motivations}

In this paper, we focus on the asymptotic behaviour of these wavenumbers, or equivalently of the associated eigenvalues, in the regime where $|k|$ is large. This behaviour is crucial for understanding the tail of modal expansions, and has consequences in at least two directions.

The first one is theoretical. Precise information on the large-wavenumber regime is needed in order to establish well-posedness results for elastic plates with prescribed source terms or boundary data \cite{bonnetier3,niclas4}. To the best of our knowledge, such results cannot be obtained directly by standard Lax-Milgram arguments or by elementary Fredholm alternatives, due to the structure of the guided-wave formulation and the nature of the modal decomposition \cite{baronian1,baronian2}. A possible strategy for proving existence is to construct solutions through modal expansions. In this approach, the decay of the modal coefficients and the asymptotic distribution of the large wavenumbers must be controlled precisely in order to prove that the resulting solution belongs to the desired Sobolev space \cite{bonnetier3}.

The second motivation is numerical. Computing the roots of the Rayleigh-Lamb dispersion relations can be delicate, especially in the large-wavenumber regime \cite{pusenjak1}. However, accurate simulations of wave propagation in elastic layers often require a precise approximation of the modal expansion, in particular near sources or boundary singularities, where the contribution of high-order modes may not be negligible \cite{castaings1}. Several numerical strategies, such as SAFE methods, have been developed to avoid a direct root-by-root computation of the dispersion relation \cite{pagneux3,hayashi1,gravenkamp1}. Nevertheless, a rigorous asymptotic expansion of the large wavenumbers, together with explicit and controlled error bounds, would make it possible to replace the computation of high-order roots by explicit asymptotic formulas beyond a suitable threshold.

For these reasons, it is essential to obtain precise asymptotic estimates for the admissible wavenumbers.

\subsection{Contributions and outline of the paper}

The present work addresses two limitations of the existing asymptotic theory. Early asymptotic formulas were given in \cite{merkulov1}. However, the article is very difficult to find online, and very concise: the formulas are stated without proofs, and only the leading-order behaviour is provided. In particular, the estimates do not include explicit constants in the remainder terms, which limits their direct use for certified numerical approximations. Despite this lack of a complete proof, these asymptotic formulas have been widely used as a reference point in the Lamb-wave community \cite{akian1,bonnet3,pagneux3,zakharov1,bonnetier3}. One of the purposes of the present work is therefore to provide a rigorous derivation of these asymptotic behaviours, with higher-order terms and explicit error estimates.

Another limitation of the existing asymptotic theory is that it mainly concerns elastic waveguides with free boundary conditions. Yet in many applications, and especially in industrial configurations, boundary conditions may be more complex. Clamped boundaries \cite{alshits1,yoon1,xu1} or boundaries coupled with an acoustic fluid \cite{inoue1,kauffmann1} arise naturally in realistic models. At present, there is no unified theoretical framework allowing one to treat these cases at the same level of precision as the classical free-boundary case. In this paper, we extend the asymptotic analysis to clamped and fluid-loaded elastic waveguides, thereby covering a broader range of physically relevant situations.

The article is organized as follows. In Section~2, we recall, in the case of Neumann boundary conditions, the mathematical framework leading to the relevant eigenvalue problem and to the associated family of wavenumbers. We then extend this formulation to clamped boundary conditions and to fluid-loaded boundaries. In Section~3, we prove the large-eigenvalue asymptotic expansions in each boundary configuration and validate the estimates numerically. Finally, in Section~4, we use these asymptotic results to establish well-posedness theorems in the two cases that, to the best of our knowledge, were not previously covered: clamped elastic waveguides and fluid-loaded elastic waveguides.

\section{Spectral problem associated to Lamb modes}

In this section, we introduce the Lamb operators that arise naturally when studying the elastodynamic equations in elastic waveguides. A standard approach consists in diagonalizing these operators in order to express the displacement field as an expansion over Lamb modes. Each Lamb mode is associated with a spectral parameter, referred to here as a Lamb eigenvalue or, equivalently, as an admissible wavenumber. We also recall the corresponding dispersion relations satisfied by these wavenumbers, which will be the starting point for the asymptotic analysis carried out in the following sections.

Throughout this paper, we restrict ourselves to isotropic elastic waveguides. We consider three types of boundary conditions, chosen because they cover the most common configurations encountered in applications. Nevertheless, the approach developed below is not specific to these cases and can be adapted, with minor modifications, to other boundary conditions when needed.

\subsection{Neumann boundary conditions}

We consider a two-dimensional infinite straight elastic waveguide $
\Omega=\{(x,z)\in \mathbb{R}\times(-h,h)\}
$ of height \(2h>0\) and density \(\rho>0\). The displacement field in the
waveguide is denoted by \(\ub=(u,v)\). Given a frequency \(\omega\in\mathbb{R}\)
and Lamé parameters \((\lambda,\mu)\), the wavefield \(\ub\) in the time-harmonic regime satisfies the elastodynamic equation
\begin{equation} \label{3_lamb1}
\nabla\cdot\sib(\ub)+\rho\omega^2\ub=-\fb
\qquad \text{in } \Omega,
\end{equation}
where \(\fb=(f_1,f_2)\) is a given source termn and the stress tensor
\(\sib(\ub)\) is given by
\begin{equation}
\sib(\ub)=
\begin{pmatrix}
(\lambda+2\mu)\partial_x u+\lambda\partial_z v
&
\mu(\partial_z u+\partial_x v)
\\
\mu(\partial_z u+\partial_x v)
&
\lambda\partial_x u+(\lambda+2\mu)\partial_z v
\end{pmatrix}
:=
\begin{pmatrix}
s & t\\
t & r
\end{pmatrix}.
\end{equation}

In this subsection, we assume that Neumann boundary conditions are imposed on
both boundaries of the waveguide:
\begin{equation} \label{3_neumann}
\sib(\ub)\cdot\bm{\nu}=\bb^\top
\qquad \text{on } \partial\Omega_\top,
\qquad
\sib(\ub)\cdot\bm{\nu}=\bb^\bot
\qquad \text{on } \partial\Omega_\bot,
\end{equation}
where \(\bb^\top=(b_1^\top,b_2^\top)\) and
\(\bb^\bot=(b_1^\bot,b_2^\bot)\) are given boundary source terms, and
\(\bm{\nu}\) denotes the outward unit normal vector to \(\Omega\).
The geometry is shown in Figure~\ref{3_guide2D}.

\begin{figure}[h]
\begin{center}
\begin{tikzpicture}[scale=0.8]
\draw (-5.5,0.7) -- (5,0.7);
\draw (-5.5,-0.7) -- (5,-0.7);
\draw [white,fill=red!10] (-5.5,0.7) -- (5,0.7) -- (5,-0.7) --(-5.5,-0.7) -- cycle;
\draw [white,fill=gray!40] (0,0.2) circle (0.4);
\draw (0,0.2) node{$\fb$};
\draw (4,0) node{$\Omega$};
\draw (5.1,0.7) node[right]{$\partial \Omega_\top$};
\draw (5.1,-0.7) node[right]{$\partial \Omega_\bot$};
\draw (-5.5,0.7) node[left]{$h$}; 
\draw (-5.5,-0.7) node[left]{$-h$}; 
\draw [ultra thick][->] (-5,-0.3)--(-4.3,-0.3) node[right]{$e_x$}; 
\draw [ultra thick][->] (-5,-0.3)--(-5,0.4) node[right]{$e_z$}; 
\draw [white,fill=gray!40] (-3.5,0.6)--(-3.5,0.8)--(-2.5,0.8)--(-2.5,0.6)--(-3.5,0.6); 
\draw (-3,0.8) node[above]{$\bb^\top$}; 
\draw [white,fill=gray!40] (-1,-0.6)--(-1,-0.8)--(2,-0.8)--(2,-0.6)--(-1,-0.6); 
\draw (0.5,-0.8) node[below]{$\bb^\bot$}; 
\end{tikzpicture}
\end{center}\vspace{-8mm}
\caption{\label{3_guide2D} Parametrization of a two-dimensional waveguide \(\Omega\) with Neumann boundary conditions. Elastic wavefields are generated by an internal source term \(\fb\) and by boundary source terms \(\bb^\top\) and \(\bb^\bot\).}
\end{figure}
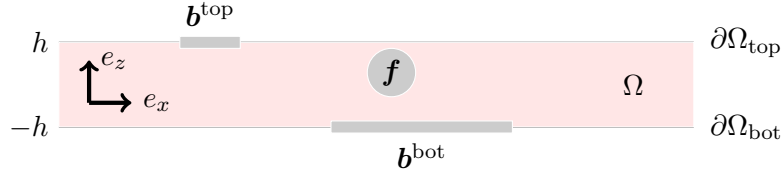

In~\cite{maupin1}, this equation is studied in operator form,
\[
\bm{Z}=\mathcal{L}(\bm{Z}),
\qquad \text{with } \bm{Z}=(u,t,-s,v).
\]
This approach was later adapted in~\cite{pagneux1} to introduce the so-called
\(\Xb/\Yb\) formulation. We define the variables
\begin{equation}\label{3_XY}
\Xb = u\,e_x+t\,e_z,
\qquad
\Yb = -s\,e_x+v\,e_z,
\end{equation}
in terms of which the elasticity equation can be recast as follows:
\begin{proposition}\label{3_equivalenceXY}
The system \eqref{3_lamb1}, with the Neumann boundary conditions \eqref{3_neumann}, is equivalent to 
\begin{equation}\label{3_eqXY}
 \partial_x\left(\begin{array}{c} \bm{X} \\ \bm{Y} \end{array}\right)=
\mathcal{L}(\Xb,\Yb)
+\left(\begin{array}{c} 0 \\ -f_2 -b_2^\top \delta_{z=h}-b_2^\bot\delta_{z=-h} \\ f_1+ b_1^\top \delta_{z=h}+b_1^\bot\delta_{z=-h} \\ 0 \end{array} \right) \quad \text{ in } \Omega,
\end{equation}
with the boundary condition $B_1(\bm{X})=B_2(\bm{Y})=0$, where $\mathcal{L}(\Xb,\Yb) = ( F(\bm{Y}), G(\bm{X}))$
and 
$F$, $G$, $B_1$ and $B_2$ are differential matrix operators defined by 
\begin{equation}\label{3_FG}
F=\left(\begin{array}{cc} -\displaystyle\frac{1}{\lambda+2\mu} & -\displaystyle\frac{\lambda}{\lambda+2\mu}\partial_z \\ \displaystyle\frac{\lambda}{\lambda+2\mu}\partial_z & -\rho\omega^2-\displaystyle\frac{4\mu(\lambda+\mu)}{\lambda+2\mu}\partial^2_{zz} \end{array}\right), \qquad G=\left(\begin{array}{cc}
\rho\omega^2 & \partial_z \\ -\partial_z & \displaystyle\frac{1}{\mu} 
\end{array}\right),
\end{equation}
\begin{equation}\label{3_B1B2}
B_1(\bm{X})=\bm{X}\cdot e_z, \qquad B_2(\bm{Y})=-\frac{\lambda}{\lambda+2\mu}\bm{Y}\cdot e_x +\frac{4\mu(\lambda+\mu)}{\lambda+2\mu}\partial_z\bm{Y}\cdot e_z. 
\end{equation}
\end{proposition}

The proof of this proposition can be found in~\cite{bonnetier3}.
In this formulation, the operators \(F\) and \(G\) depend only on the transverse
variable \(z\) and are defined on a single cross-section of the waveguide,
whereas derivatives with respect to \(x\) appear only on the left-hand side
of~\eqref{3_eqXY}. We consider the space
\begin{equation}
H_{\mathrm{neum}}
:=
\left\{
(\Xb,\Yb)\in \bigl(H^2(-h,h)\bigr)^4
\,\middle|\,
B_1(\Xb)(\pm h)=B_2(\Yb)(\pm h)=0
\right\},
\end{equation}
and the operator
\begin{equation}
\mathcal{L}_{\mathrm{neum}}:
\begin{array}{rcl}
H_{\mathrm{neum}} & \longrightarrow & \bigl(L^2(-h,h)\bigr)^4, \\[0.2em]
(\Xb,\Yb) & \longmapsto & \bigl(F(\Yb),G(\Xb)\bigr).
\end{array}
\end{equation}
To diagonalize this operator, it is customary to introduce the Lamb modes:

\begin{defi}\label{3_lambmode}
A Lamb mode with Neumann boundary conditions is a non-trivial 
\((\Xb,\Yb)\in H_{\mathrm{neum}}\) associated with a wavenumber
\(k\in\mathbb{C}\), satisfying
\[
\mathcal{L}_{\mathrm{neum}}(\Xb,\Yb)=ik(\Xb,\Yb).
\]
\end{defi}

These eigenmodes do not form a Hilbert basis. However, using the formalism
presented above, the following result was proved in~\cite{akian1}:

\begin{theorem}\label{3_th2}
For almost every frequency \(\omega\in\mathbb{R}_+\), the Lamb modes form a
complete set of functions in \(H_{\mathrm{neum}}\).
\end{theorem}

\begin{rem}
As explained in the introduction, Lamb modes can also be defined as
separated-variable solutions of~\eqref{3_lamb1}, of the form
\[
\ub(x,z)=(
\phi(z),
\psi(z))
\exp(ikx).
\]
Both formalisms lead to the same dispersion relations for the wavenumbers
\(k\), as well as to the same expressions for the Lamb modes. However, to the
best of our knowledge, the separated-variable formulation does not provide a
framework in which the completeness of the Lamb modes can be rigorously
justified.
\end{rem}

These Lamb modes have already been extensively studied. The next proposition
recalls the dispersion equations satisfied by the wavenumber \(k\), known as
the Rayleigh-Lamb equations in the case of Neumann boundary conditions. The
proof can be found in~\cite{achenbach1,royer1}. We first introduce the
transverse and longitudinal wavenumbers
\begin{equation}\label{kt_kl}
k_t=\omega\sqrt{\frac{\rho}{\mu}},
\qquad
k_\ell=\omega \sqrt{\frac{\rho}{\lambda+2\mu}},
\end{equation}
which satisfy \(k_t > k_\ell\).

\begin{proposition}
Every wavenumber \(k\in \mathbb{C}\) associated with a Lamb mode satisfying
Neumann boundary conditions satisfies either the symmetric dispersion relation,
also called the symmetric Rayleigh-Lamb equation,
\begin{equation} \label{3_disps}
p^2=k_\ell^2-k^2,
\qquad
q^2=k_t^2-k^2,
\qquad
\left(q^2-k^2\right)^2\tan(qh)=-4k^2pq\tan(ph),
\end{equation}
or the antisymmetric dispersion relation, also called the antisymmetric
Rayleigh-Lamb equation,
\begin{equation}\label{3_dispa}
p^2=k_\ell^2-k^2,
\qquad
q^2=k_t^2-k^2,
\qquad
\left(q^2-k^2\right)^2\tan(ph)=-4k^2pq\tan(qh).
\end{equation}
\end{proposition}

For a given value of \(k\), explicit expressions for the corresponding Lamb
modes can be found in~\cite{achenbach1} and are given in Appendix~A. 

\subsection{Countability of the solutions of the dispersion relations}

Although it seems clear that the Neumann Rayleigh-Lamb equations have only a
countable number of solutions, we found it difficult to locate a complete proof
of this fact in the literature. For the sake of a rigorous study of the
behavior of the wavenumbers in this paper, we provide in this subsection a
self-contained proof of the following result.

To prove this result, we rewrite the Rayleigh-Lamb equations as zero sets of
holomorphic functions. This reformulation will also be useful later to establish
the asymptotic behavior of these functions. Using trigonometric identities,
equation~\eqref{3_disps} can be rewritten as
\begin{equation} \label{eq:symVN2}
\dfrac{1}{2}
\left(\left(q^2 - k^2\right)^2 + 4k^2pq\right)
\sin((q+p)h)
+
\dfrac{1}{2}
\left(\left(q^2 - k^2\right)^2 - 4k^2pq\right)
\sin((q-p)h)
= 0,
\end{equation}
and, similarly, equation~\eqref{3_dispa} becomes
\begin{equation}\label{eq:antisymVN2}
\dfrac{1}{2}
\left(\left(q^2 - k^2\right)^2 + 4k^2pq\right)
\sin((q+p)h)
-
\dfrac{1}{2}
\left(\left(q^2 - k^2\right)^2 - 4k^2pq\right)
\sin((q-p)h)
= 0.
\end{equation}

For any complex number \(z\), we define
$
\sqrt{z}=z^{1/2}=\sqrt{|z|}e^{i\mathrm{arg}(z)/2}$ with 
$\mathrm{arg}(z)\in(-\pi,\pi],
$
using the principal branch of the complex square root. We observe that \(p\)
and \(q\) are defined up to a multiplicative sign. However, the choice of signs
for \(p\) and \(q\) does not affect the corresponding value of \(k\). Thus, we can arbitrarily 
choose
\begin{equation}
q = ik\left(1 - \dfrac{k_t^2}{k^2}\right)^{1/2},
\qquad
p = ik\left(1 - \dfrac{k_\ell^2}{k^2}\right)^{1/2}.
\end{equation}
We then introduce the functions
\begin{equation}
\begin{array}{rrcl}
q :& \mathbb{C}^* & \to & \mathbb{C}\\
& z & \mapsto & iz\left(1 - \dfrac{k_t^2}{z^2}\right)^{1/2}
\end{array}
\quad \text{and} \quad
\begin{array}{rrcl}
p :& \mathbb{C}^* & \to & \mathbb{C}\\
& z & \mapsto & iz\left(1 - \dfrac{k_\ell^2}{z^2}\right)^{1/2}
\end{array},
\end{equation}
with an extension at \(0\) given by \(q(0)=k_t\) and \(p(0)=k_\ell\), and we
also define

\begin{equation}
\begin{array}{rrcl}
f_s :& \mathbb{C} & \to & \mathbb{C}\\
& z & \mapsto &
\left(q^2(z) - z^2\right)^2\sin(q(z)h)\cos(p(z)h)
+
4z^2p(z)q(z)\sin(p(z)h)\cos(q(z)h)
\end{array},
\end{equation}
and
\begin{equation}
\begin{array}{rrcl}
f_a :& \mathbb{C} & \to & \mathbb{C}\\
& z & \mapsto &
\left(q^2(z) - z^2\right)^2\sin(p(z)h)\cos(q(z)h)
+
4z^2p(z)q(z)\sin(q(z)h)\cos(p(z)h)
\end{array}.
\end{equation}

The functions \(f_s\) and \(f_a\) are holomorphic on
\(\mathbb{C}\setminus[-k_t,k_t]\), and their zeros are precisely the solutions
of~\eqref{3_disps} and~\eqref{3_dispa}, respectively. These functions will be
used below to prove the proposition, and again later in the analysis.

\begin{proposition}
The set of solutions of the Neumann Rayleigh-Lamb equations
\eqref{3_disps} and~\eqref{3_dispa} is at most countable. Moreover, if
\(k\in\mathbb{C}\) is a solution of~\eqref{3_disps}
(resp.~\eqref{3_dispa}), then \(-k\) and \(\overline{k}\) are also
solutions of~\eqref{3_disps} (resp.~\eqref{3_dispa}).
\end{proposition}

\begin{proof}
We prove the result for \(f_s\), the proof for \(f_a\) being identical.
Let \(Z(f_s)\) denote the set of zeros of \(f_s\). Since \(f_s\) is
holomorphic on \(\mathbb{C}\setminus[-k_t,k_t]\), the principle of isolated
zeros implies that \(Z(f_s)\) has no accumulation point in
\(\mathbb{C}\setminus[-k_t,k_t]\).

For \(n\in\mathbb{N}^*\), set
\begin{equation}
K_n =
\left\{
z\in\mathbb{C}
\,\middle|\,
|z|\leq n
\quad\text{and}\quad
\operatorname{dist}\bigl(z,[-k_t,k_t]\bigr)\geq \frac{1}{n}
\right\}.
\end{equation}

Then \(K_n\) is compact and
$
\mathbb{C}\setminus[-k_t,k_t]
=
\bigcup_{n\geq 1} K_n.
$
Moreover, \(Z(f_s)\cap K_n\) is compact. If \(Z(f_s)\cap K_n\) were infinite,
then by the Bolzano-Weierstrass theorem it would have an accumulation point, contradicting the principle
of isolated zeros. Hence \(Z(f_s)\cap K_n\) is finite for every
\(n\in\mathbb{N}^*\), and therefore
$
Z(f_s)\cap\bigl(\mathbb{C}\setminus[-k_t,k_t]\bigr)
=
\bigcup_{n\geq 1} \bigl(Z(f_s)\cap K_n\bigr)
$
is at most countable.

It remains to examine the behavior on the cut \([-k_t,k_t]\). Let
\(z=re^{i\theta}\), with \(r\in\mathbb{R}_+\) and
\(\theta\in[0,2\pi)\), and define
$
s(z)=\sqrt{r}e^{i\theta/2}.$
This defines another branch of the square root, which is holomorphic on
\(\mathbb{C}\setminus\mathbb{R}_+\) and agrees with the principal branch on
\(\mathbb{R}_-\). Let \(\widetilde q\) and \(\widetilde p\) denote the
functions defined as \(q\) and \(p\), respectively, but using \(s\) instead of
the principal branch of the square root. Then \(\widetilde q\) is holomorphic
on
$
\mathbb{C}^*
\setminus
\bigl((-\infty,-k_t]\cup[k_t,+\infty)\bigr),
$
and \(\widetilde p\) is holomorphic on
$
\mathbb{C}^*
\setminus
\bigl((-\infty,-k_\ell]\cup[k_\ell,+\infty)\bigr).
$

We now define \(\widetilde f_{s,1}\) as the function obtained from \(f_s\) by
replacing \(q\) with \(\widetilde q\). The function \(\widetilde f_{s,1}\) is
holomorphic on
$
\mathbb{C}
\setminus
\bigl(
(-\infty,-k_t]\cup[k_t,+\infty)\cup[-k_\ell,k_\ell]
\bigr),
$
and coincides with \(f_s\) on
$
(-k_t,-k_\ell)\cup(k_\ell,k_t).
$
By the same argument as above, the zero set of \(\widetilde f_{s,1}\) in its
domain of holomorphy is at most countable. Hence the set of zeros of \(f_s\)
on \((-k_t,-k_\ell)\cup(k_\ell,k_t)\) is at most countable.

Finally, let \(\widetilde f_{s,2}\) be the function obtained from \(f_s\) by
replacing both \(q\) and \(p\) with \(\widetilde q\) and \(\widetilde p\). This
function is holomorphic on
$
\mathbb{C}^*
\setminus
\bigl((-\infty,-k_\ell]\cup[k_\ell,+\infty)\bigr)
$
and coincides with \(f_s\) on \((-k_\ell,k_\ell)\setminus \{0\}\). Again, its zero set is at
most countable, and therefore so is the set of zeros of \(f_s\) on
\((-k_\ell,k_\ell) \setminus \{0\}\).

Combining these different subsets of \(Z(f_s)\), together with the finite set
\(\{-k_t,-k_\ell,0,k_\ell,k_t\}\), we conclude that \(Z(f_s)\) is at most
countable. The same argument applies to \(f_a\), which proves the first part
of the proposition.

We now prove the symmetry properties. First, \(f_s\) is an odd function. Next, we use the identities
\begin{equation}
\forall z\in\mathbb{C}\setminus\mathbb{R}_-,
\qquad
\sqrt{\overline{z}}=\overline{\sqrt{z}},
\qquad
\forall z\in\mathbb{C}\setminus\mathbb{R},
\qquad
1-\frac{k_t^2}{z^2}\in\mathbb{C}\setminus\mathbb{R}.
\end{equation}
It follows that
$
\overline{q(z)}=-q(\overline{z})$ and $\overline{p(z)}=-p(\overline{z}).
$
Consequently,
$
f_s(\overline{z})=-\overline{f_s(z)},
$
which concludes the proof. 
\end{proof}

Thanks to the countability of the solutions of the dispersion relations, the
Lamb wavenumbers can be indexed as \((k_n)_{n\in\mathbb{N}}\), and the
corresponding Lamb modes as
$
(\Xb_n,\Yb_n)=(u_n,t_n,-s_n,v_n).$
This makes it possible to expand the solution of~\eqref{3_lamb1} as
\begin{equation}
(\ub,\vb)
\approx
\sum_{n\in\mathbb{N}}
\bigl(a_n(x)u_n(y),\, b_n(x)v_n(y)\bigr).
\end{equation}

Substituting this expansion into~\eqref{3_XY}, or equivalently
into~\eqref{3_lamb1}, yields a family of one-dimensional equations satisfied
by the modal amplitudes \(a_n\) and \(b_n\), which can then be solved as
in~\cite{bonnetier3,pagneux1}. However, in order to establish a complete well-posedness
result, derive an explicit representation formula for the solution, and obtain
estimates with respect to the source terms, it is first necessary to gain a
precise understanding of the asymptotic behavior of the Lamb wavenumbers
\(k_n\) as \(n\to\infty\). This will be the subject of the next section. Before doing so, we present the main differences induced by changing the
boundary conditions.

\subsection{Dirichlet boundary conditions}

We now consider the same waveguide, but replace the Neumann boundary
conditions~\eqref{3_neumann} with Dirichlet boundary conditions
\begin{equation} \label{3_dirichlet}
\ub=\bb^\top
\quad \text{on } \partial\Omega_\top,
\qquad
\ub=\bb^\bot
\quad \text{on } \partial\Omega_\bot,
\end{equation}
where \(\bb^\top=(b_1^\top,b_2^\top)\) and
\(\bb^\bot=(b_1^\bot,b_2^\bot)\) are given boundary source terms.
Let \(\bb\) be a continuous lifting of \(\bb^\top\) and \(\bb^\bot\) to
\(\Omega\), and define
$
\gb=\nabla\cdot\sib(\bb)+\rho\omega^2\bb.
$
Using the same \(\Xb/\Yb\) formalism as in the previous subsection, the
elasticity equation can be rewritten as follows:

\begin{proposition}
The system~\eqref{3_lamb1}, with the Dirichlet boundary
conditions~\eqref{3_dirichlet}, is equivalent to
\begin{equation}
 \partial_x
 \begin{pmatrix}
 \Xb\\
 \Yb
 \end{pmatrix}
=
\mathcal{L}(\Xb,\Yb)
+
\begin{pmatrix}
0\\
-f_2-g_2\\
f_1+g_1\\
0
\end{pmatrix}
\qquad \text{in } \Omega,
\end{equation}
with the boundary conditions
$
B_3(\Xb)=
B_4(\Yb)=0,$
where \(\mathcal{L}\) is defined by~\eqref{3_FG}, and \(B_3\) and \(B_4\)
are the operators
\begin{equation}\label{3_B3B4}
B_3(\Xb)=\Xb\cdot e_x,
\qquad
B_4(\Yb)=\Yb\cdot e_z.
\end{equation}
\end{proposition}

As in the Neumann case, we introduce the space
\begin{equation}
H_{\mathrm{diri}}
:=
\left\{
(\Xb,\Yb)\in \bigl(H^2(-h,h)\bigr)^4
\,\middle|\,
B_3(\Xb)(\pm h)=B_4(\Yb)(\pm h)=0
\right\},
\end{equation}
and the operator
\begin{equation}
\mathcal{L}_{\mathrm{diri}}:
\begin{array}{rcl}
H_{\mathrm{diri}} & \longrightarrow & \bigl(L^2(-h,h)\bigr)^4,\\[0.2em]
(\Xb,\Yb) & \longmapsto & \bigl(F(\Yb),G(\Xb)\bigr).
\end{array}
\end{equation}
To diagonalize this operator, we again introduce Lamb modes.

\begin{defi}\label{3_lambmode_diri}
A Lamb mode with Dirichlet boundary conditions is a non-trivial
\((\Xb,\Yb)\in H_{\mathrm{diri}}\) associated with a wavenumber
\(k\in\mathbb{C}\), satisfying
\[
\mathcal{L}_{\mathrm{diri}}(\Xb,\Yb)
=
ik\,(\Xb,\Yb).
\]
\end{defi}

Adapting the arguments of~\cite{akian1,lutianov1}, one can show that these
eigenmodes form a complete set of functions for the operator
\(\mathcal{L}_{\mathrm{diri}}\). Similarly, adapting the derivations
of~\cite{achenbach1,royer1}, one obtains the dispersion relations satisfied by
the Dirichlet Lamb modes:

\begin{proposition}
Every wavenumber \(k\in\mathbb{C}\) associated with a Lamb mode satisfying
Dirichlet boundary conditions satisfies either the symmetric Dirichlet
dispersion relation
\begin{equation} \label{eq:symD1}
p^2 = k_\ell^2-k^2,
\qquad
q^2 = k_t^2-k^2,
\qquad
k^2\tan(qh)+pq\tan(ph)=0,
\end{equation}
or the antisymmetric Dirichlet dispersion relation
\begin{equation}\label{eq:antisymD1}
p^2 = k_\ell^2-k^2,
\qquad
q^2 = k_t^2-k^2,
\qquad
k^2\tan(ph)+pq\tan(qh)=0.
\end{equation}
\end{proposition}

Explicit expressions for the corresponding Lamb modes can also be found in Appendix~A. Furthermore, the proof developed in
Section~2.2 can be adapted to these dispersion relations to show that
their sets of solutions are at most countable. The resulting wavenumbers present
the same symmetry properties with respect to the real axis and the origin as in
the Neumann case.

\subsection{Fluid boundary conditions}

We now turn to fluid boundary conditions, which
constitute the third most common class of boundary conditions encountered in
the study of elastic waveguides. In this setting, the waveguide is
immersed in an acoustic fluid of density \(\rho_f\) and sound speed \(c_f\).
The fluid domain above the waveguide is denoted by
$
\Omega_f^\top=\mathbb{R}\times(h,+\infty),
$
while the fluid domain below the waveguide is denoted by
$
\Omega_f^\bot=\mathbb{R}\times(-\infty,-h).
$
The geometry is represented in Figure~\ref{3_guide2D_fluide}.

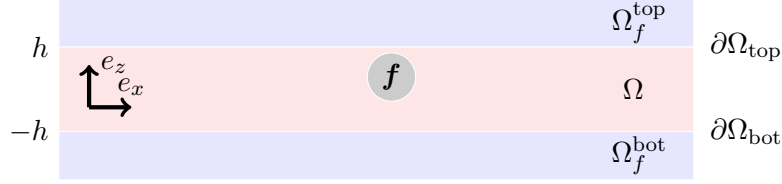
\begin{figure}[h]
\begin{center}\begin{tikzpicture}[scale=0.8]\draw (-5.5,0.7) -- (5,0.7);\draw (-5.5,-0.7) -- (5,-0.7);\draw [white,fill=red!10] (-5.5,0.7) -- (5,0.7) -- (5,-0.7) --(-5.5,-0.7) -- cycle;\draw [white,fill=blue!10] (-5.5,-1.5) -- (5,-1.5) -- (5,-0.7) --(-5.5,-0.7) -- cycle;\draw [white,fill=blue!10] (-5.5,1.5) -- (5,1.5) -- (5,0.7) --(-5.5,0.7) -- cycle;\draw [white,fill=gray!40] (0,0.2) circle (0.4);\draw (0,0.2) node{$\fb$};\draw (4,0) node{$\Omega$};\draw (4.1,1.1) node{$\Omega_f^\top$};\draw (4.1,-1.1) node{$\Omega_f^\bot$};\draw (5.1,0.7) node[right]{$\partial \Omega_\top$};\draw (5.1,-0.7) node[right]{$\partial \Omega_\bot$};\draw (-5.5,0.7) node[left]{$h$};\draw (-5.5,-0.7) node[left]{$-h$};\draw [ultra thick][->] (-5,-0.3)--(-4.3,-0.3) node[above]{$e_x$};\draw [ultra thick][->] (-5,-0.3)--(-5,0.4) node[right]{$e_z$};\end{tikzpicture}\end{center}\vspace{-8mm}
\caption{\label{3_guide2D_fluide}
Geometry of a two-dimensional elastic waveguide \(\Omega\) immersed in an
acoustic fluid. Elastic wavefields are generated by an internal source term
\(\fb\).}
\end{figure}

The displacement field in the solid is denoted by
\(\ub=(u,v)\) and satisfies equation~\eqref{3_lamb1}. In the fluid domains
\(\Omega_f^j\), \(j\in\{\top,\bot\}\), we denote by
\(\ub_f^j=(u_f^j,v_f^j)\) the fluid displacement and by \(p_f^j\) the acoustic
pressure. The fluid motion is governed by the standard acoustic equations~\cite{diaz1}:
\begin{equation}\label{eq_fluid}
\left\{
\begin{array}{l}
\nabla p_f^j-\omega^2\rho_f \ub_f^j = 0, \\[0.2em]
p_f^j +\rho_f c_f^2\nabla\cdot \ub_f^j=0,
\end{array}
\right.
\qquad \text{in } \Omega_f^j.
\end{equation}

At the interfaces \(\partial\Omega_\top\) and \(\partial\Omega_\bot\), we
impose continuity of the normal displacement together with continuity of the
normal stress. This leads to the following transmission conditions, referred
to hereafter as fluid boundary conditions:
\begin{equation}\label{bc_fluid}
\forall j\in\{\bot,\top\},
\qquad
\ub\cdot\bm{\nu}=\ub_f^j\cdot\bm{\nu},
\qquad
\sib(\ub)\cdot \bm{\nu}=-p_f^j\bm{\nu},
\qquad
\text{on } \partial\Omega_j .
\end{equation}

Although, to our knowledge, this coupled fluid-solid problem has not been
formulated previously in the operator framework considered here, most of the
analysis developed for the elastic part can be reused. Keeping the
\(\Xb/\Yb\) variables introduced above, we define the additional fluid
variables
\begin{equation}
\bm{P}^j
=
p_f^j\,e_x
+
\partial_x p_f^j\,e_z .
\end{equation}
The coupled system can then be rewritten as follows: 

\begin{proposition}
The coupled system consisting of~\eqref{eq_fluid} and~\eqref{3_lamb1}, with the fluid boundary conditions~\eqref{bc_fluid}, is
equivalent to
\begin{equation}
\partial_x
\begin{pmatrix}
\bm{P}^j\\
\Xb\\
\Yb
\end{pmatrix}
=
\begin{pmatrix}
\mathcal{F}(\bm{P}^j)\\
\mathcal{L}(\Xb,\Yb)
\end{pmatrix}
+
\begin{pmatrix}
\bm{0}_{\mathbb{R}^3}\\
-f_2\\
f_1\\
0
\end{pmatrix},
\end{equation}
with the boundary conditions
\begin{equation}
B_1(\Xb)=0,
\qquad
B_2(\Yb)=B_5(\bm{P}^j),
\qquad
B_3(\Xb)=B_6(\bm{P}^j),
\qquad
B_4(\Yb)=B_7(\bm{P}^j),
\qquad
\text{on } \partial\Omega_j,
\end{equation}
for \(j\in\{\top,\bot\}\). Here, \(\mathcal{L}\) is defined by~\eqref{3_FG},
\(B_1,B_2,B_3,B_4\) are defined in
\eqref{3_B1B2} and \eqref{3_B3B4}, and
\(\mathcal{F}\), \(B_5\), \(B_6\), and \(B_7\) are given by
\begin{equation}
\mathcal{F}
=
\begin{pmatrix}
0 & 1\\
-\partial_{yy}-\dfrac{\omega^2}{c_f^2} & 0
\end{pmatrix},
\qquad
B_5(\bm{P}^j)
=
-(\nu\cdot e_z)\,\bm{P}^j\cdot e_x,
\end{equation}
and
\begin{equation}
B_6(\bm{P}^j)
=
\frac{1}{\omega^2\rho_f}\,\bm{P}^j\cdot e_x,
\qquad
B_7(\bm{P}^j)
=
\frac{1}{\omega^2\rho_f}\,
\partial_y\bigl(\bm{P}^j\cdot e_x\bigr).
\end{equation}
\end{proposition}

Following the Neumann and Dirichlet cases, we introduce the space
\begin{multline}
H_{\mathrm{flu}}
:=
\Bigl\{
(\bm{P}^\top,\bm{P}^\bot,\Xb,\Yb)
\in \bigl(H^2(\mathbb{R})\bigr)^8
\,\Big|\,
B_1(\Xb)(\pm h)=0,
\quad
B_2(\Yb)(\pm h)=B_5(\bm{P}^j)(\pm h),
\\
B_3(\Xb)(\pm h)=B_6(\bm{P}^j)(\pm h),
\quad
B_4(\Yb)(\pm h)=B_7(\bm{P}^j)(\pm h)
\Bigr\},
\end{multline}
and the operator
\begin{equation}
\mathcal{L}_{\mathrm{flu}}:
\begin{array}{rcl}
H_{\mathrm{flu}}
& \longrightarrow &
\bigl(L^2(-h,h)\bigr)^8,
\\[0.2em]
(\bm{P}^\top,\bm{P}^\bot,\Xb,\Yb)
& \longmapsto &
\bigl(
\mathcal{F}(\bm{P}^\top),
\mathcal{F}(\bm{P}^\bot),
F(\Yb),
G(\Xb)
\bigr).
\end{array}
\end{equation}

To diagonalize this operator, we again introduce Lamb modes:

\begin{defi}\label{3_lambmode_flu}
A Lamb mode with fluid boundary conditions is a non-trivial
$
(\bm{P}^\top,\bm{P}^\bot,\Xb,\Yb)
\in H_{\mathrm{flu}}
$
associated with a wavenumber \(k\in\mathbb{C}\) and satisfying
\[
\mathcal{L}_{\mathrm{flu}}
(\bm{P}^\top,\bm{P}^\bot,\Xb,\Yb)
=
ik(\bm{P}^\top,\bm{P}^\bot,\Xb,\Yb).
\]
\end{defi}

Adapting the arguments of~\cite{akian1}, these eigenmodes form
a complete set of functions for the operator
\(\mathcal{L}_{\mathrm{flu}}\). Because of the additional fluid variables
\(\bm{P}^j\), the proof does not follow directly from the Neumann case.
However, since the main analytical difficulties remain concentrated in the
elastic part of the problem, the techniques developed
in~\cite{akian1} can be extended by combining them with standard estimates for
the acoustic field in the fluid, which has already been extensively studied
(see, for instance, \cite{niclas4}).

More importantly for our purposes, this framework provides a systematic way to
derive the dispersion relations satisfied by the wavenumbers. Such relations
have been computed in~\cite{osborne1}, and we only state the final result here.
As in~\eqref{kt_kl}, we introduce the fluid wavenumber
\begin{equation}
k_f=\frac{\omega}{c_f}.
\end{equation}
\begin{proposition}
Every wavenumber \(k\in\mathbb{C}\) associated with a Lamb mode satisfying
fluid boundary conditions satisfies either the symmetric fluid dispersion
relation
\begin{multline}\label{eq:symFN1}
p^2 = k_\ell^2-k^2,
\qquad
q^2 = k_t^2-k^2,
\qquad
d^2 = k_f^2-k^2,
\\
\left(q^2-k^2\right)^2\sin(qh)\cos(ph)
+
4k^2pq\,\sin(ph)\cos(qh)
+
i\frac{\rho_f k_t^4}{\rho}\frac{p}{d}\sin(qh)\sin(ph)
=
0,
\end{multline}
or the antisymmetric fluid dispersion relation
\begin{multline}\label{eq:antisymFN1}
p^2 = k_\ell^2-k^2,
\qquad
q^2 = k_t^2-k^2,
\qquad
d^2 = k_f^2-k^2,
\\
\left(q^2-k^2\right)^2\sin(ph)\cos(qh)
+
4k^2pq\,\sin(qh)\cos(ph)
-
i\frac{\rho_f k_t^4}{\rho}\frac{p}{d}\cos(qh)\cos(ph)
=
0.
\end{multline}

\end{proposition}

Following the methodology developped in~\cite{achenbach1}, one can find explicit expressions for the correspond Lamb modes, presented in Appendix~A. 

The proof developed in Section~2.2 can also be adapted to these
dispersion relations to show that their sets of solutions are at most
countable, by introducing
\begin{equation}
d
=
ik\left(1-\frac{k_f^2}{k^2}\right)^{1/2}.
\end{equation}
The symmetry with respect to the real axis is preserved. However, unlike the
Neumann and Dirichlet cases, the symmetry with respect to the origin is lost:
if \(k\) is a solution, then \(-k\) is not necessarily a solution.

More generally, the framework developed in this section can be adapted to a
wide variety of boundary conditions. For instance, one may consider a fluid
boundary condition on one side of the waveguide and a Neumann boundary
condition on the other, leading to so-called leaky waveguides~\cite{gravenkamp1,merlini1}. The same operator-based approach can then be used to
derive the corresponding dispersion relations.

\section{Asymptotic behavior of the wavenumbers}

Having established how the elastic displacement can be decomposed into a sum
of eigenmodes associated with wavenumbers depending on the boundary
conditions, we now turn to the construction of modal expansions of the
solution. As explained in~\cite{bonnetier3}, proving the
well-posedness of the elastic problem and justifying the modal decomposition
require a precise understanding of the asymptotic behavior of the wavenumbers
\(k_n\) as \(|k_n|\to+\infty\). More specifically, asymptotic expansions of
the wavenumbers are needed to control the decay of the modal coefficients and,
consequently, the convergence of the modal Lamb series.

In the Neumann case, an asymptotic expansion of the wavenumbers was proposed
in~\cite{merkulov1}. However, the result is stated without proof. Although this
expansion has been validated numerically in numerous studies (see, for
instance,~\cite{pagneux3}), a rigorous derivation still appears to be missing.
Moreover, the constants involved in the expansion are not made explicit
in~\cite{merkulov1}, whereas such information is essential for obtaining
quantitative estimates on the decay of the modal expansion. Finally, similar
asymptotic results do not seem to be available for the other boundary
conditions considered in the previous section.

For all these reasons, we provide in this section a rigorous derivation of the
asymptotic behavior of the solutions of the various dispersion relations.

\subsection{Neumann boundary conditions}

We begin with the Neumann case and establish the following asymptotic
expansions:

\begin{theorem}
The solutions \(k_n\) of the symmetric dispersion relation~\eqref{3_disps}
satisfying
$
\Re(k_n)\ge 0$ and $
\Im(k_n)\ge 0,$
admit the following asymptotic expansion as \(n\to\infty\): 
\begin{equation}
\label{eq:asympksymVN}
k_n
=
\frac{1}{2h}\ln(4\pi n)
-\frac{1}{8hn}
+
\frac{i}{2h}
\left(
\pi\left(2n-\frac12\right)
-\frac{\ln(4\pi n)}{2\pi n}
-\frac{h^2(k_\ell^2+k_t^2)}{2\pi n}
\right)
+
o\!\left(\frac1n\right)
\end{equation}
Similarly, the solutions \(k_n\) of the antisymmetric dispersion
relation~\eqref{3_dispa} satisfying
$
\Re(k_n)\ge 0,$ and
$
\Im(k_n)\ge 0,
$
admit the asymptotic expansion
\begin{equation}
\label{eq:asympkantisymVN}
k_n
=
\frac{1}{2h}\ln(4\pi n)
-\frac{3}{8hn}
+
\frac{i}{2h}
\left(
\pi\left(2n-\frac32\right)
-\frac{\ln(4\pi n)}{2\pi n}
-\frac{h^2(k_\ell^2+k_t^2)}{2\pi n}
\right)
+
o\!\left(\frac1n\right).
\end{equation}
\end{theorem}

\begin{proof}

Let \(k\) be a solution of~\eqref{3_disps}, rewritten as~\eqref{eq:symVN2}. Throughout the proof, all asymptotic expansions are understood in the limit \(|k|\to+\infty\). The Taylor expansion of the square root near \(1\) yields
\begin{equation}
q = ik - \dfrac{ik_t^2}{2k} + o\left(\dfrac{1}{|k|^2}\right),
\qquad
p = ik - \dfrac{ik_\ell^2}{2k} + o\left(\dfrac{1}{|k|^2}\right).
\end{equation}
We then obtain
\begin{equation} \label{eq:premierfacteur}
\frac{1}{2}\left(\left(q^2-k^2\right)^2+4k^2pq\right)
=
\left(k_\ell^2-k_t^2\right)k^2+o(|k|),
\end{equation}
and
\begin{equation} \label{eq:secondfacteur}
\frac{1}{2}\left(\left(q^2-k^2\right)^2-4k^2pq\right)
=
4k^4-\left(k_\ell^2+3k_t^2\right)k^2+o(|k|).
\end{equation}
Using the Taylor expansion of the sine function near \(0\),
\begin{equation} \label{eq:DLsin(p-q)h}
\sin\bigl((q-p)h\bigr)
=
ih\,\dfrac{k_\ell^2-k_t^2}{2k}
+
o\left(\dfrac{1}{|k|^2}\right),
\end{equation}
and therefore, multiplying \eqref{eq:secondfacteur} and \eqref{eq:DLsin(p-q)h},
\begin{equation} \label{eq:secondterme}
\dfrac{1}{2}
\left(\left(q^2-k^2\right)^2-4k^2pq\right)
\sin((q-p)h)
=
2ih\left(k_\ell^2-k_t^2\right)k^3
+
o(|k|^2).
\end{equation}
Using a trigonometric identity,
\begin{multline}
\sin((p+q)h)
=
\sin(2ikh)
\cos\left(
ih\dfrac{k_\ell^2+k_t^2}{2k}
+
o\left(\dfrac{1}{|k|^2}\right)
\right)
\\
-
\cos(2ikh)
\sin\left(
ih\dfrac{k_\ell^2+k_t^2}{2k}
+
o\left(\dfrac{1}{|k|^2}\right)
\right).
\end{multline}
From this, we deduce
\begin{equation} \label{eq:DLsin(p+q)h}
\sin((p+q)h)
=
\sin(2ikh)
\left(
1+o\left(\dfrac{1}{|k|}\right)
\right)
-
\cos(2ikh)
\left(
ih\dfrac{k_\ell^2+k_t^2}{2k}
+
o\left(\dfrac{1}{|k|^2}\right)
\right).
\end{equation}
Multiplying \eqref{eq:premierfacteur} and \eqref{eq:DLsin(p+q)h}, we obtain
\begin{multline} \label{eq:premierterme}
\dfrac{1}{2}
\left(\left(q^2-k^2\right)^2+4k^2pq\right)
\sin((q+p)h)
=
\left(
\left(k_\ell^2-k_t^2\right)k^2
+
o(|k|)
\right)
\sin(2ikh)
\\
-
\left(
\dfrac{ih}{2}
\left(k_\ell^2+k_t^2\right)
\left(k_\ell^2-k_t^2\right)k
+
o(1)
\right)
\cos(2ikh).
\end{multline}
Adding \eqref{eq:secondterme} and \eqref{eq:premierterme}, using the identities
\(\sin(iz)=i\sinh(z)\) and \(\cos(iz)=\cosh(z)\), and dividing by
\(i(k_\ell^2-k_t^2)k^2\), which is non-zero, equation~\eqref{eq:symVN2}
becomes
\begin{equation} \label{eq:DLdeSVN1}
\left(
1+o\left(\dfrac{1}{|k|}\right)
\right)
\sinh(2kh)
-
\left(
h\dfrac{k_\ell^2+k_t^2}{2k}
+
o\left(\dfrac{1}{|k^2|}\right)
\right)
\cosh(2kh)
+
2kh
+
o(1)
=
0.
\end{equation}
If \(\Re(k)\) remains bounded, then so do \(\cosh(2kh)\) and
\(\sinh(2kh)\), and equation~\eqref{eq:DLdeSVN1} shows that no solutions can
exist. We therefore look for solutions satisfying $\Re(k)\geq 0$ and \(\Re(k)\to+\infty\). 
In
this regime, $
\sinh(2kh)=\frac{e^{2kh}}{2}+o(1)$ and 
$\cosh(2kh)=\frac{e^{2kh}}{2}+o(1),$ hence \eqref{eq:DLdeSVN1} becomes
\begin{equation}
\label{eq:DLdeSVN2}
\left(
1+o\left(\dfrac{1}{|k|}\right)
\right)e^{2kh}
-
\left(
h\dfrac{k_\ell^2+k_t^2}{2k}
+
o\left(\dfrac{1}{|k|}\right)
\right)e^{2kh}
+
4kh
+
o(1)
=
0.
\end{equation}
Using a standard asymptotic localization argument for the zeros of
transcendental equations (see~\cite[Chap.~11]{olver5}),
equation~\eqref{eq:DLdeSVN2} can be simplified as
\begin{equation}
\label{DLdeSVN}
e^{2kh}
+
o\bigl(|e^{2kh}|\bigr)
+
4kh
=
0.
\end{equation}

Let \(w=2kh\). Equation~\eqref{DLdeSVN} shows that asymptotic solutions of
\eqref{eq:symVN2} are close to the asymptotic solutions of
\begin{equation}
e^w+2w=0
\qquad\Longleftrightarrow\qquad
-we^{-w}=\frac12.
\end{equation}
The solutions of this equation are given by the Lambert \(W\) function and are
\begin{equation}
\{-W_n(1/2)\; ;\; n\in\mathbb Z\,\},
\end{equation}
where \(W_n\), \(n\in\mathbb Z\), denotes the different branches of the
Lambert \(W\) function (see a definition in~\cite{corless1}). From
\cite{corless1}, we have
\begin{equation}
\label{eq:DLLambert}
W_n(z)
\underset{|n|\to\infty}{=}
\ln(z)-\ln(2i\pi n)+2i\pi n+o(1),
\end{equation}
where \(\ln\) denotes the principal branch of the complex logarithm.
Using the symmetries of the dispersion relation, we assume that
\(\Re(k)\ge0\) and \(\Im(k)\ge0\), which corresponds to taking \(n\le0\).
Setting \(N=-n\), and $w_N=-W_{-N}(1/2)$, we obtain
\begin{equation}
e^{w_N}
=
e^{-W_{-N}(1/2)}
\underset{N\to\infty}{=}
\exp\left(
\ln(4\pi N)
+
i\pi\left(2N-\frac12\right)
+
o(1)
\right)
=
-4i\pi N+o(N).
\end{equation}
Hence \(e^{w_N}+2w_N=o(N)\), showing that \(k_N=w_N/(2h)\) is indeed an
asymptotic solution of~\eqref{DLdeSVN}.
Knowing this first-order expansion of \(k_N\), we have
\(o(|e^{2kh}|)=o(|k|)\), and equation~\eqref{eq:DLdeSVN2} becomes, after the
change of variable \(w=2kh\),
\begin{equation}
e^w
-
h^2\dfrac{k_\ell^2+k_t^2}{w}e^w
+
2w
+
o(1)
=
0.
\end{equation}
To improve the asymptotic expansion of \(k_N\), let \(w_N\) denote the
solutions of this equation. We write
\begin{equation}
w_N
=
\ln(4\pi N)
+
2i\pi N
-
\frac{i\pi}{2}
+
v_N,
\qquad
v_N
\underset{N\to\infty}{=}
o(1).
\end{equation}
We also introduce
\[
\beta=h^2\left(k_\ell^2+k_t^2\right).
\]
Substituting this expression into the previous equation yields
\begin{equation}
-4i\pi N e^{v_N}
+
2\beta e^{v_N}(1+o(1))
+
4i\pi N
+
2\ln(4\pi N)
-
i\pi
+
o(1)
=
0.
\end{equation}
Since \(o(e^{v_N})=o(1)\), we obtain
\begin{equation}
\label{eq:expvn}
e^{v_N}
=
\dfrac{
-4i\pi N
-2\ln(4\pi N)
+i\pi
+o(1)
}{
-4i\pi N
+2\beta
}
=
1
-
i\dfrac{\ln(4\pi N)}{2\pi N}
-
i\dfrac{\beta}{2\pi N}
-
\dfrac{1}{4N}
+
o\left(\dfrac{1}{N}\right).
\end{equation}
Since \(e^{v_N}=1+v_N+o(|v_N|)\),
\begin{equation}
v_N
=
-i\dfrac{\ln(4\pi N)}{2\pi N}
+
o(|v_N|)
=
-i\dfrac{\ln N}{2\pi N}
+
o\left(\dfrac{\ln N}{N}\right).
\end{equation}
Using the fact that \(v_N^2=o(1/N)\) and that 
$e^{v_N}
=
1+v_N+\frac{v_N^2}{2}+o(|v_N|^2),
$
and reinjecting into~\eqref{eq:expvn}, we finally obtain
\begin{equation}
v_N
=
-i\dfrac{\ln(4\pi N)}{2\pi N}
-
i\dfrac{\beta}{2\pi N}
-
\dfrac{1}{4N}
+
o\left(\dfrac{1}{N}\right),
\end{equation}
hence
\begin{equation}
\label{dev_wn_sym}
w_N
\underset{N\to\infty}{=}
\ln(4\pi N)
-
\dfrac{1}{4N}
+
i\left(
\pi\left(2N-\frac12\right)
-
\dfrac{\ln(4\pi N)}{2\pi N}
-
\dfrac{h^2(k_\ell^2+k_t^2)}{2\pi N}
\right)
+
o\left(\dfrac{1}{N}\right).
\end{equation}
Similarly, in the antisymmetric case, the same computations applied to
equation~\eqref{3_dispa}, rewritten as~\eqref{eq:antisymVN2}, yield
\begin{equation} \label{eq:DLdeAVN1}
\left(
1+o\left(\dfrac{1}{|k|}\right)
\right)e^{2kh}
+
\left(
h\dfrac{k_\ell^2+k_t^2}{2k}
+
o\left(\dfrac{1}{|k|}\right)
\right)e^{2kh}
-
4kh
+
o(1)
=
0,
\end{equation}
and therefore
\begin{equation}
\label{DLdeAVN}
e^{2kh}
+
o\bigl(|e^{2kh}|\bigr)
-
4kh
=0.
\end{equation}

The solutions are asymptotically close to those of the equation
\begin{equation}
-we^{-w}
=
-\dfrac{1}{2},
\end{equation}
and we have that 
\begin{equation}
w_N
\underset{N\to\infty}{=}
\ln(4\pi N)
+
i\pi\left(2N-\frac{3}{2}\right)
+
o(1).
\end{equation}
Equation~\eqref{eq:DLdeAVN1} then becomes
\begin{equation}
\label{eq:DLdeAVN3}
e^{2kh}
-
h\dfrac{k_\ell^2+k_t^2}{2k}e^{2kh}
-
4kh
+
o(1)
=
0,
\end{equation}
and the same computations as in the symmetric case yield
\begin{equation}
w_N
\underset{N\to\infty}{=}
\ln(4\pi N)
-
\dfrac{3}{4N}
+
i\left(
\pi\left(2N-\frac{3}{2}\right)
-
\dfrac{\ln(4\pi N)}{2\pi N}
-
\dfrac{h^2(k_\ell^2+k_t^2)}{2\pi N}
\right)
+
o\left(\dfrac{1}{N}\right).
\end{equation}
\end{proof}

\begin{rem}
The asymptotic expansion proposed in~\cite{merkulov1} for the
symmetric modes is
\begin{equation}
w_N
\underset{N\to\infty}{=}
\ln\left(
2\pi\left(2N-\frac{1}{2}\right)
\right)
+
i\left(
\pi\left(2N-\frac{1}{2}\right)
-
\dfrac{
\ln\bigl(2\pi(2N-\tfrac12)\bigr)
}{
\pi(2N-\tfrac12)
}
\right)
+
O\left(\frac{1}{w_N}\right).
\end{equation}
Since
\[
\ln\left(2\pi\left(2N-\frac12\right)\right)
=
\ln(4\pi N)
-
\frac{1}{4N}
+
o\left(\frac{1}{N}\right),
\]
we see that our expansion~\eqref{dev_wn_sym} is consistent with the result
of~\cite{merkulov1}. However, we emphasize that our expansion
contains one additional order, together with explicit constants in the
\(O(1/N)\) term. In principle, the method can be continued to derive higher-order
terms in the asymptotic expansion. Indeed, once an expansion of \(w_N\) is
known up to a given order, it can be reinjected into the dispersion relation
to determine the next correction term. The calculations, however, become
rapidly cumbersome and the resulting expressions increasingly intricate. Since
the \(O(1/N)\) term is sufficient for the applications considered in this
paper, we do not pursue the expansion further.
\end{rem}

\subsection{Other boundary conditions}

We now apply the same reasoning to derive asymptotic expansions for the
wavenumbers in the case of Dirichlet boundary conditions.

\begin{theorem}
Let
\begin{equation}
\alpha
=
2\,\dfrac{k_t^2-k_\ell^2}{k_\ell^2+k_t^2}.
\end{equation}
The solutions \(k_n\) of the symmetric dispersion relation~\eqref{eq:symD1}
satisfying \(\Re(k_n)\geq 0\) and \(\Im(k_n)\geq 0\) admit the asymptotic
expansion
\begin{equation}\label{eq:asympksymD}
k_n=
\frac{1}{2h}\ln(2\alpha\pi n)
-
\frac{3}{8hn}
+
\frac{i}{2h}
\left(
\pi\left(2n-\frac{3}{2}\right)
-
\frac{\ln(2\alpha\pi n)}{2\pi n}
-
\frac{h^2(k_\ell^2+k_t^2)}{2\pi n}
\right)
+
o\left(\frac{1}{n}\right).
\end{equation}
The solutions \(k_n\) of the antisymmetric dispersion relation~\eqref{eq:antisymD1}
satisfying \(\Re(k_n)\geq 0\) and \(\Im(k_n)\geq 0\) admit the asymptotic
expansion
\begin{equation}\label{eq:asympkantisymD}
k_n=
\frac{1}{2h}\ln(2\alpha\pi n)
-
\frac{1}{8hn}
+
\frac{i}{2h}
\left(
\pi\left(2n-\frac{1}{2}\right)
-
\frac{\ln(2\alpha\pi n)}{2\pi n}
-
\frac{h^2(k_\ell^2+k_t^2)}{2\pi n}
\right)
+
o\left(\frac{1}{n}\right).
\end{equation}
\end{theorem}

\begin{proof}
The proof follows the same steps as in the Neumann case. Expanding \(p\) and \(q\), the analogue of equation~\eqref{eq:DLdeSVN1}
becomes
\begin{equation}
\label{eq:DLSD1}
\left(
1+o\left(\frac{1}{|k|}\right)
\right)
\sinh(2kh)
-
h\frac{k_\ell^2+k_t^2}{2k}
\cosh(2kh)
+
\alpha kh
+
o(1)
=
0.
\end{equation}
Again, the real part of the solutions must tend to \(+\infty\), and the
analogue of equation~\eqref{DLdeSVN} is
\begin{equation}
\label{eq:DLSD2}
e^{2kh}
+
o\left(|e^{2kh}|\right)
+
2\alpha kh
=
0.
\end{equation}
The conclusion then follows by repeating the same computations as in the
Neumann case.
\end{proof}

We now turn to fluid boundary conditions. Comparing the fluid dispersion
relation~\eqref{eq:symFN1} with the Neumann relation~\eqref{3_disps}, we note
that the first two terms are the same. It therefore remains to analyze the
additional fluid term. A direct asymptotic analysis shows that this term is
negligible compared with the first two. Thus, the asymptotic expansions of the
solutions of~\eqref{eq:symFN1} and~\eqref{eq:antisymFN1} are the same as those
obtained from~\eqref{eq:symVN2} and~\eqref{eq:antisymVN2}.

\begin{theorem}
The solutions \(k_n\) of the symmetric dispersion relation~\eqref{eq:symFN1}
satisfying \(\Im(k_n)\geq 0\) admit the asymptotic expansion
\begin{equation}
\label{eq:asympksymF}
k_n= 
\pm \frac{1}{2h}\ln(4\pi n)
\mp
\frac{1}{8hn}
+
\frac{i}{2h}
\left(
\pi\left(2n-\frac{1}{2}\right)
-
\frac{\ln(4\pi n)}{2\pi n}
-
\frac{h^2(k_\ell^2+k_t^2)}{2\pi n}
\right)
+
o\left(\frac{1}{n}\right).
\end{equation} 
The solutions \(k_n\) of the antisymmetric dispersion relation~\eqref{eq:antisymFN1}
satisfying \(\Im(k_n)\geq 0\) admit the asymptotic expansion
\begin{equation}
\label{eq:asympkantisymF}
k_n=
\pm \frac{1}{2h}\ln(4\pi n)
\mp
\frac{3}{8hn}
+
\frac{i}{2h}
\left(
\pi\left(2n-\frac{3}{2}\right)
-
\frac{\ln(4\pi n)}{2\pi n}
-
\frac{h^2(k_\ell^2+k_t^2)}{2\pi n}
\right)
+
o\left(\frac{1}{n}\right).
\end{equation}
\end{theorem}

\subsection{Numerical validation}

We now perform numerical simulations to compute the solutions of the dispersion
relations and compare them with the asymptotic formulas derived in the previous
subsections. The solutions of the Rayleigh-Lamb equations can be obtained
using standard root-finding algorithms~\cite{pusenjak1}, or more efficiently
through SAFE methods~\cite{pagneux3,gravenkamp1}. The simulations are carried out for a duralumin plate with
\(h=1\) mm, \(\rho=2.79\) mg/mm\(^3\),
\(f=5\) MHz (and therefore \(\omega=2\pi f\)),
\(c_\ell=\omega/k_\ell=6.4\) mm/\(\mu\)s, and
\(c_t=\omega/k_t=3.1\) mm/\(\mu\)s.

Figure~\ref{fig:k_vide_neuman} displays the numerically computed solutions and the asymptotic curves obtained in the case of a free plate in vacuum with
Neumann boundary conditions. Two zoom levels are provided for different ranges of the imaginary part. The agreement between the numerical solutions and the asymptotic approximation becomes visually excellent as \(\Im(k)\) and \(n\)
increase.

The same analysis is presented in Figure~\ref{fig:k_dirichlet} for a clamped
duralumin plate with Dirichlet boundary conditions, and in
Figure~\ref{fig:k_fluide_neuman} for a free duralumin plate immersed in water
with fluid boundary conditions. For the fluid, we take
\(\rho_f=1.00\) mg/mm\(^3\) and \(c_f=1.5\) mm/\(\mu\)s.
In this case, the symmetry with respect to the imaginary axis is lost, and it
is therefore necessary to display both sides of the spectrum to assess the
quality of the asymptotic approximation. In all cases, the same excellent
visual agreement is observed.

\begin{figure}[htbp]
\centering \hspace{-0.5cm}
\input{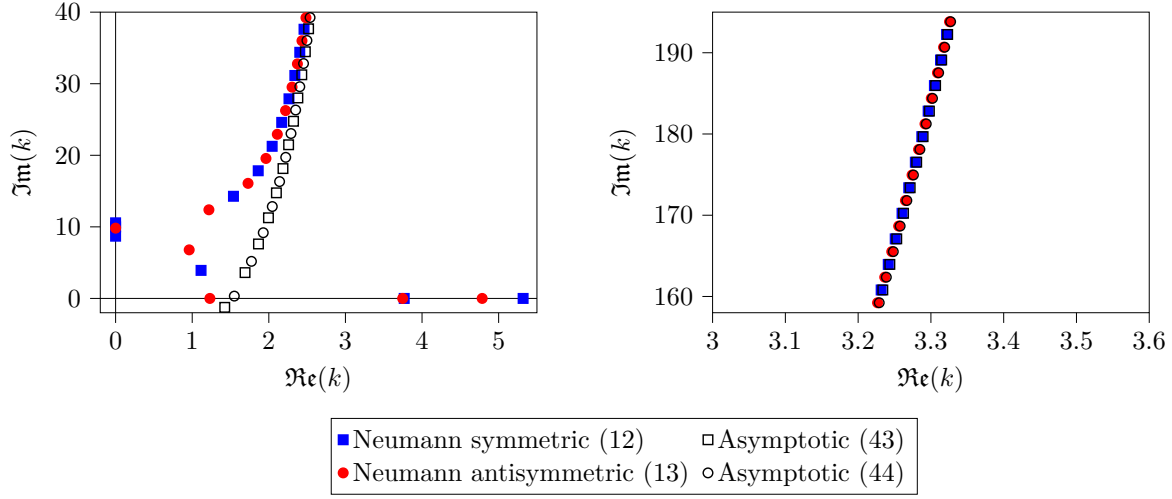}
\caption{Comparison between the numerical solutions of the dispersion relation
and the asymptotic approximation in the Neumann boundary condition case for a
duralumin plate. Left: zoom for small values of \(\Im(k)\). Right: zoom for
large values of \(\Im(k)\).}
\label{fig:k_vide_neuman}
\end{figure}

\begin{figure}[htbp]
\centering
\input{dirichlet}
\caption{Comparison between the numerical solutions of the dispersion relation
and the asymptotic approximation in the Dirichlet boundary condition case for a
duralumin plate. Left: zoom for small values of \(\Im(k)\). Right: zoom for
large values of \(\Im(k)\).}
\label{fig:k_dirichlet}
\end{figure}

\begin{figure}[htbp]
\centering
\input{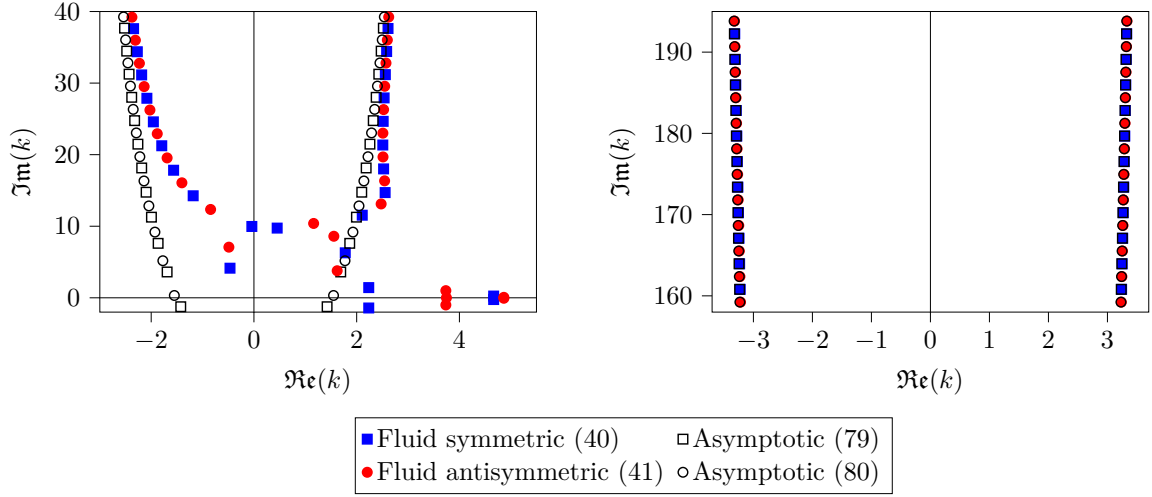}
\caption{Comparison between the numerical solutions of the dispersion relation
and the asymptotic approximation in the fluid boundary condition case for a
duralumin plate immersed in water. Left: zoom for small values of
\(\Im(k)\). Right: zoom for large values of \(\Im(k)\).}
\label{fig:k_fluide_neuman}
\end{figure}

After visually confirming the accuracy of the asymptotic approximation, we
investigate the quality of the asymptotic expansion itself, and in particular
the validity of the \(o(1/n)\) remainder. To this end, we plot in
Figure~\ref{fig:asymptote} the quantity
\[
\ln\bigl(n|k_n-k_{a,n}|\bigr)
\]
as a function of \(\ln(n)\), where \(k_n\) denotes the numerically computed
solution and \(k_{a,n}\) the corresponding asymptotic approximation. The plots
are generated for \(1000\leq n\leq2500\). As expected, we observe that \(n|k_n-k_{a,n}|\) tends to zero in all cases,
thereby validating the additional \(1/n\) term derived in the asymptotic
expansions. Moreover, comparison with a reference slope equal to \(-1\)
suggests that the next term in the asymptotic expansion may be of order
\(1/n^2\), with a coefficient depending on both the boundary conditions and
the symmetry class. 

\begin{figure}[htbp]
\centering
\input{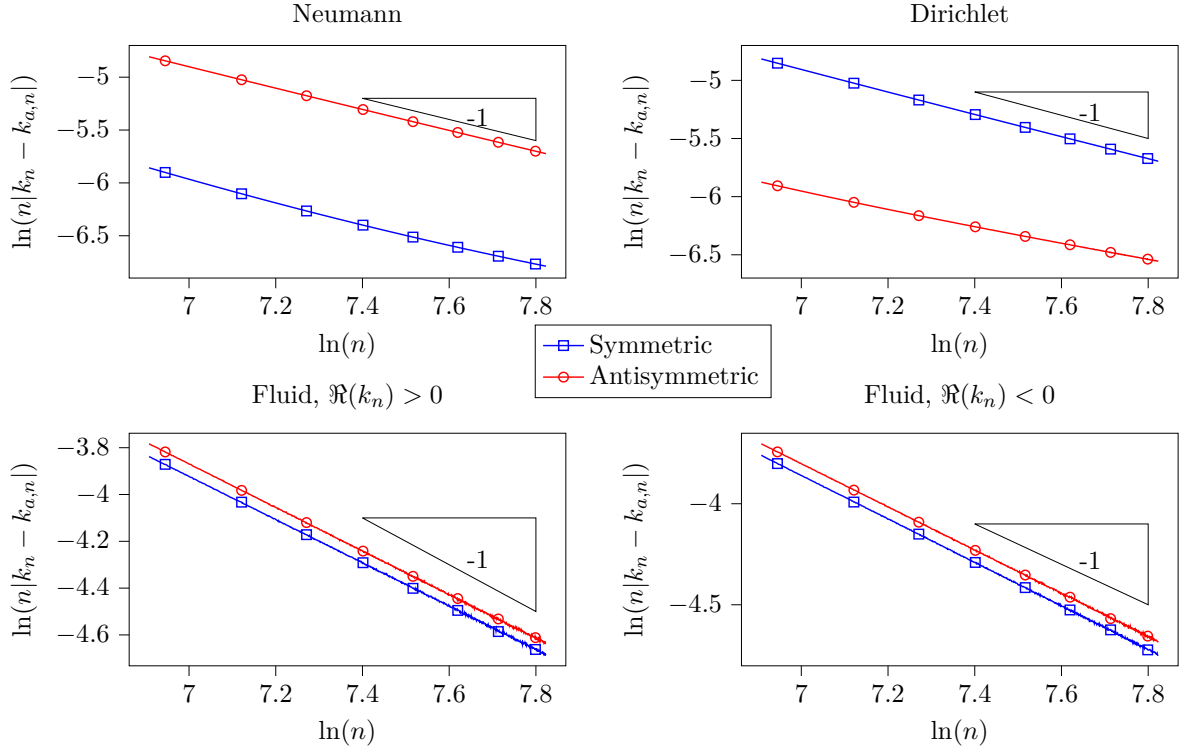}
\caption{Evolution of \(\ln(n|k_n-k_{a,n}|)\) as a function of \(\ln(n)\),
where \(k_n\) denotes the numerical solution of the dispersion relation and
\(k_{a,n}\) the asymptotic approximation. Top left: Neumann boundary
conditions. Top right: Dirichlet boundary conditions. Bottom left: fluid
boundary conditions, branch with \(\Re(k_n)>0\). Bottom right: fluid boundary
conditions, branch with \(\Re(k_n)<0\). In each case, a line of slope \(-1\)
is shown for comparison.}
\label{fig:asymptote}
\end{figure}

%

\section{Applications to well-posedness theorems}

To conclude this article, we would like to emphasize how the asymptotic
expansions derived in the previous sections can be used to establish
well-posedness theorems and quantitative estimates controlling the solution in
terms of the source data for elastic wave equations with various boundary
conditions.

As explained in~\cite{bonnetier3} in the Neumann case, asymptotic expansions of
the wavenumbers \(k_n\) play a crucial role in proving well-posedness results
and deriving modal representations of the elastic displacement in an infinite
waveguide. In the aforementioned work, some of the arguments relied on
asymptotic properties of the wavenumbers that were not fully justified, since
only the first terms of the expansion were available and no precise control of
the remainder was known. Thanks to the results established in the present
paper, these asymptotic expansions are now rigorously justified.

Rather than revisiting the Neumann case, we illustrate the usefulness of our
approach by considering the Dirichlet case, which has not yet been treated in the literature. The corresponding
well-posedness result can be stated as follows.

\begin{theorem}

Let \(r>0\). For almost every \(\omega\in\mathbb{R}_+\),
\(\fb=(f_1,f_2)\in L^2(\Omega_r)\), and
\(\bb^\top=(b_1^\top,b_2^\top)\),
\(\bb^\bot=(b_1^\bot,b_2^\bot)\in H^{3/2}(-r,r)\),
the elasticity equation~\eqref{3_lamb1} with the Dirichlet
boundary conditions~\eqref{3_dirichlet} and an appropriate radiation
condition (see~\cite{bonnetier3}), admits a unique solution
\(\ub\in H^2_{\mathrm{loc}}(\Omega)\).
Moreover, this solution admits the Lamb-mode decomposition
\begin{equation}\label{3_sol2D}
u(x,z)
=
\sum_{n>0} a_n(x)u_n(z),
\qquad
v(x,z)
=
\sum_{n>0} b_n(x)v_n(z),
\end{equation}
where
$
a_n
=
G_1^n\ast F_1^n
-
G_2^n\ast F_2^n$ and
$
b_n
=
G_2^n\ast F_1^n
-
G_1^n\ast F_2^n, 
$
with
\begin{equation}\label{3_gi}
G_1^n(x)
=
\frac{1}{2}e^{ik_n|x|},
\qquad
G_2^n(x)
=
\frac{x}{2|x|}
e^{ik_n|x|},
\end{equation}
and
\begin{equation}\label{3_F1}
F_1^n(x)
=
\frac{1}{\langle \Xb_n,\Yb_n\rangle}
\int_{-h}^{h}
(f_1+g_1)(x,z)\,u_n(z)\,\mathrm{d}z,
\qquad
F_2^n(x)
=
\frac{1}{\langle \Xb_n,\Yb_n\rangle}
\int_{-h}^{h}
(f_2+g_2)(x,z)\,u_n(z)\,\mathrm{d}z,
\end{equation}
where
$
\gb
=
\nabla\cdot\sib(\bb)
+
\rho\omega^2\bb,
$
and \(\bb\) denotes a continuous lifting of \(\bb^\top\) and
\(\bb^\bot\) to \(\Omega\).

Furthermore, there exists a constant \(C>0\), depending only on \(h\),
\(\omega\), and \(r\), such that
\begin{equation}\label{control}
\|\ub\|_{H^2(\Omega_r)}
\leq
C
\left(
\|\fb\|_{L^2(\Omega)}
+
\|\bb^\top\|_{H^{3/2}(\mathbb{R})}
+
\|\bb^\bot\|_{H^{3/2}(\mathbb{R})}
\right).
\end{equation}

\end{theorem}

\begin{proof}
The proof follows exactly the same strategy as the proof presented in~\cite{bonnetier3},
and we therefore only highlight the differences between the Neumann and
Dirichlet cases.

The first step consists in proving uniqueness of the solution. This is done by
showing that all modal coefficients in the Lamb-mode decomposition must vanish
when the source terms vanish. Projecting the equation onto each Lamb mode then
shows that any solution necessarily admits the modal decomposition stated in
the theorem. The main difficulty is to prove that the proposed modal decomposition is
well-defined, that is, that the corresponding series converges. To this end,
one has to study the behaviour of the products \(a_nu_n\) and \(b_nv_n\) as
\(n\to\infty\), using the asymptotic expansions established in the previous
section.

Using the expressions of \(u_n\) and \(v_n\) given in Appendix~A, we
obtain
\begin{equation}
u_n(z)
\sim
\frac{-i\pi}{4h}(c_q-c_p)(h-|z|)
\exp\left(
\left(\frac{|z|}{h}+1\right)
\frac{\ln(2\alpha\pi n)}{2}
\right),
\end{equation}
and
\begin{equation}
v_n(z)
\sim
\frac{-i\pi}{4h}(c_q-c_p)
\left(
\frac{hz}{|z|}-z
\right)
\exp\left(
\left(\frac{|z|}{h}+1\right)
\frac{\ln(2\alpha\pi n)}{2}
\right).
\end{equation}
Taking the \(L^2\)-norm yields
\begin{equation}
\|u_n\|_{L^2(-h,h)},
\;
\|v_n\|_{L^2(-h,h)}
\sim
\frac{\pi^2\alpha h^{1/2}n|c_q-c_p|}
{\ln(2\alpha\pi n)^{3/2}}.
\end{equation}
Similarly, using the \(L^\infty(-h,h)\)-norm, we obtain
\begin{equation}
\|u_n\|_{L^\infty(-h,h)},
\;
\|v_n\|_{L^\infty(-h,h)}
\leq
\frac{\pi^{1+2\alpha}}{4}|c_q-c_p|\,n.
\end{equation}
Next, using the expressions of \(s_n\) and \(t_n\) given in Appendix~A, we compute
\begin{equation}
\langle \Xb_n,\Yb_n\rangle
=
\int_{-h}^{h}
\bigl(t_nv_n-s_nu_n\bigr)
=
ik_n\mu A_1A_2
+
ik_n\mu B_1B_2
-
ik_nq_nC_1C_2
-
ik_nq_nD_1D_2,
\end{equation}
where
\begin{multline}
A_1=-p_nq_n\cos(q_nh)+k_n^2\sin(q_nh),
\qquad
B_1=q_n^2\cos(p_nh)-k_n^2\sin(p_nh),
\\
A_2=
-\sin(q_nh)
\left(
h-\frac{1}{2p_n}\sin(2p_nh)
\right)
+
\sin(p_nh)
\left(
\frac{\sin((q_n-p_n)h)}{q_n-p_n}
-
\frac{\sin((q_n+p_n)h)}{p_n+q_n}
\right),
\\
B_2=
\sin(p_nh)
\left(
h-\frac{1}{2q_n}\sin(2q_nh)
\right)
+
\sin(q_nh)
\left(
-\frac{\sin((q_n-p_n)h)}{q_n-p_n}
+
\frac{\sin((q_n+p_n)h)}{p_n+q_n}
\right),
\\
C_1=
(\lambda+2\mu)q_n\cos(q_nh)
-
\lambda p_n\sin(q_nh),
\qquad
D_1=
-(\lambda+2\mu)q_n\cos(p_nh)
+
\lambda q_n\sin(p_nh),
\\
C_2=
\cos(q_nh)
\left(
h+\frac{1}{2p_n}\sin(2p_nh)
\right)
-
\cos(p_nh)
\left(
\frac{\sin((p_n-q_n)h)}{p_n-q_n}
+
\frac{\sin((q_n+p_n)h)}{q_n+p_n}
\right),
\\
D_2=
-\cos(p_nh)
\left(
h+\frac{1}{2q_n}\sin(2q_nh)
\right)
+
\cos(q_nh)
\left(
\frac{\sin((p_n-q_n)h)}{p_n-q_n}
+
\frac{\sin((q_n+p_n)h)}{q_n+p_n}
\right).
\end{multline}
Using the asymptotic expansions established previously, we obtain after
straightforward calculations
\begin{multline}
A_1=
k_n^2
\Biggl(
\left(
1-\frac{ik_t^2h}{2k_n}
\right)
\cos(ik_nh)
+
\left(
1+\frac{ik_t^2h}{2k_n}
\right)
\sin(ik_nh)
+
o(n^{-1/2})
\Biggr),
\\
B_1=
k_n^2
\Biggl(
\left(
-1+\frac{ik_\ell^2h}{2k_n}
\right)
\cos(ik_nh)
+
\left(
-1-\frac{ik_\ell^2h}{2k_n}
\right)
\sin(ik_nh)
+
o(n^{-1/2})
\Biggr), \qquad \qquad 
\\
A_2\;(\text{resp. }B_2)
=
\frac{\cos(ik_nh)}{k_n}
\frac{i(k_t^2-k_\ell^2)h^2}{2}
-
\frac{\cos(3ik_nh)}{k_n^3}
\frac{i(k_t^2-k_\ell^2)}{16}\qquad \qquad \qquad 
\\
+
\frac{\sin(ik_nh)}{k_n^2}
\left(
\frac{(k_\ell^2-k_t^2)h}{4}
+
\frac{(k_\ell^4-k_t^4)h^3}{8}
\pm
\frac{(k_\ell^2-k_t^2)^2h^3}{24}
\right)
+
o(n^{-3/2}), 
\\
C_1=
ik_n
\Biggl(
\left(
(\lambda+2\mu)
+
\lambda\frac{ik_t^2h}{2k_n}
\right)
\cos(ik_nh)
+
\left(
-\lambda
+
(\lambda+2\mu)\frac{ik_t^2h}{2k_n}
\right)
\sin(ik_nh)
+
o(n^{-1/2})
\Biggr),
\\
D_1=
ik_n
\Biggl(
\left(
-(\lambda+2\mu)
-
\lambda\frac{ik_\ell^2h}{2k_n}
\right)
\cos(ik_nh)
+
\left(
\lambda
-
(\lambda+2\mu)\frac{ik_\ell^2h}{2k_n}
\right)
\sin(ik_nh)
+
o(n^{-1/2})
\Biggr),
\\
C_2\;(\text{resp. }D_2)
=
\frac{\sin(ik_nh)}{k_n}
\frac{i(k_t^2-k_\ell^2)h^2}{2}
+
\frac{\sin(3ik_nh)}{k_n^3}
\frac{i(k_t^2-k_\ell^2)}{16}
\\
-
\frac{\cos(ik_nh)}{k_n^2}
\left(
\frac{(k_\ell^2-k_t^2)h}{4}
+
\frac{(k_\ell^4-k_t^4)h^3}{8}
\pm
\frac{(k_\ell^2-k_t^2)^2h^3}{24}
\right)
+
o(n^{-3/2}).
\end{multline}
It follows that
\begin{multline}
ik_n\mu(A_1A_2+B_1B_2)
=
-\frac{\mu\pi^3\alpha n^2h^2(k_t^2-k_\ell^2)^2(1+i)}
{12}
+
o(n^2),
\\
-ik_nq_n(C_1C_2+D_1D_2)
=
\frac{\pi^2\alpha ih^2n^2(k_t^2-k_\ell^2)^2}{2}
\left(
\frac{\lambda(i-1)}{6}
-
\frac{\mu}{3} 
\right)+
o(n^2),
\end{multline}
and therefore
\begin{equation}
\langle \Xb_n,\Yb_n\rangle
\sim
\frac{\pi^2\alpha h^2(k_t^2-k_\ell^2)^2n^2}{2}
\left(
\frac{-\lambda-\mu}{6}
-
\frac{\lambda+3\mu}{6}i
\right).
\end{equation}
We can now estimate \(a_n\):
\begin{equation}
\|a_n\|_{L^2(-r,r)}
\leq
\|G_1^n\|_{L^1(-r,r)}
\|F_1^n\|_{L^2(-r,r)}
+
\|G_2^n\|_{L^1(-r,r)}
\|F_2^n\|_{L^2(-r,r)}
=
O\left(
\frac{1}{n^2}
\right).
\end{equation}

A similar estimate holds for \(b_n\). Following the final part of the proof
presented in~\cite{bonnetier3}, we conclude that the sequences
$
\|a_n\|_{L^2(-r,r)}
\|u_n\|_{L^2(-h,h)}
$ and $
\|b_n\|_{L^2(-r,r)}
\|v_n\|_{L^2(-h,h)}$
are summable. This proves the convergence of the modal expansion and yields
estimate~\eqref{control}.
\end{proof}

\begin{rem}
The main difference with the Neumann case is that the present well-posedness
theorem requires less regularity on the source term. Indeed, using the
asymptotic estimates established above, we have shown that
$
\|a_n\|_{L^2(-r,r)}
\|u_n\|_{L^2(-h,h)}
=
O\!\left(\frac{1}{n\ln(n)^{3/2}}\right),
$
whereas, under the same regularity assumptions on \(f\), the corresponding
estimate in the Neumann case is only
$
O\!\left(\frac{1}{\ln(n)^{3/2}}\right),
$
which is not summable. It is therefore interesting to observe that the
Dirichlet boundary condition leads to a more general well-posedness result
than the Neumann boundary condition.
\end{rem} 

Finally, the same approach could be applied to the case of fluid boundary
conditions. However, considering the expressions of the Lamb modes given in
Appendix~A together with the asymptotic expansions derived in
Section~3.2, which coincide with those obtained in the Neumann case, the
proof would follow exactly the same lines as in~\cite{bonnetier3}. Consequently, one
would recover the same well-posedness theorem as in the Neumann setting.

More generally, this observation suggests that once suitable asymptotic
estimates are available for both the Lamb modes and the solutions of the
dispersion relation, the derivation of well-posedness results becomes a rather
direct consequence of the modal framework. In this sense, the main analytical
difficulty lies in establishing the asymptotic behaviour of the modes and
wavenumbers, while the well-posedness theorem itself follows with relatively
little additional effort.

\section{Conclusion}

In this work, we have provided a rigorous proof of the asymptotic behaviour
of the wavenumbers associated with Lamb modes, a result that has long been
used as a conjecture or heuristic approximation within the community. We
considered three classes of boundary conditions: the classical Neumann case,
which has been extensively studied in the literature; the Dirichlet case,
which is less common but frequently arises in industrial applications; and
finally the fluid boundary condition, which is of particular interest in both
industrial and biomedical contexts.

For each of these cases, we derived explicit asymptotic expansions of the
wavenumbers and showed how these expansions can be used to establish
well-posedness results for the corresponding elastic wave equations.
Furthermore, we demonstrated that the asymptotic behaviour of the Lamb modes
and of the solutions of the dispersion relations provides the key ingredient
for obtaining modal decompositions together with quantitative estimates on the
solution.

The methodology developed in this paper can be extended to a wide range of
configurations involving mixed boundary conditions, provided that the
corresponding dispersion relations can be analysed asymptotically. As a
result, the present work provides both a theoretical framework for the study
of elastic waveguides and a practical numerical tool. Indeed, the explicit
asymptotic formulas obtained here make it possible to approximate high-order
wavenumbers accurately without resorting to costly numerical computations,
thereby significantly accelerating modal calculations.

The approach developed here should also extend naturally to anisotropic plates.
In that setting, the Rayleigh-Lamb equations are replaced by more general
dispersion relations involving the roots of the Christoffel equation~\cite{kuznetsov1}. The
main challenge would then consist in deriving suitable asymptotic expansions
of these roots and of the associated dispersion relations. Since the present
work shows that asymptotic information on the wavenumbers is the key
ingredient for obtaining modal decompositions and well-posedness results, we
expect that a similar strategy could be successfully applied in the
anisotropic setting. This question, however, lies beyond the scope of the
present paper and will be the subject of future investigations.

\section*{Appendix A}

\subsection{Lamb modes with Neumann and fluid boundary conditions}

We recall the expressions of the Lamb modes for Neumann boundary conditions derived in~\cite{bonnetier3}. In the symmetric case, the Lamb modes $(u,t,-s,v)$ are proportional to
\begin{equation}
\left(\begin{array}{c}
ik(q^2-k^2)\sin(qh)\cos(pz)-2ikpq\sin(ph)\cos(qz) \\
2ik\mu(q^2-k^2)p\bigl(-\sin(qh)\sin(pz)+\sin(ph)\sin(qz)\bigr)\\
(q^2-k^2)\bigl((\lambda+2\mu)k^2+\lambda p^2\bigr)\sin(qh)\cos(pz)
-4\mu pq k^2\sin(ph)\cos(qz)\\
-p(q^2-k^2)\sin(qh)\sin(pz)-2k^2p\sin(ph)\sin(qz)
\end{array}\right).
\end{equation}

In the antisymmetric case, they are proportional to
\begin{equation}
\left(\begin{array}{c}
ik(q^2-k^2)\cos(qh)\sin(pz)-2ikpq\cos(ph)\sin(qz) \\
2ik\mu(q^2-k^2)p\bigl(\cos(qh)\cos(pz)-\cos(ph)\cos(qz)\bigr)\\
(q^2-k^2)\bigl((\lambda+2\mu)k^2+\lambda p^2\bigr)\cos(qh)\sin(pz)
-4\mu pq k^2\cos(ph)\sin(qz)\\
p(q^2-k^2)\cos(qh)\cos(pz)+2k^2p\cos(ph)\cos(qz)
\end{array}\right).
\end{equation}

Following the derivation initiated in~\cite{osborne1}, it can be shown that these expressions remain unchanged in the presence of fluid boundary conditions.

\subsection{Lamb modes with Dirichlet boundary conditions}

Following the methodology developed in~\cite{achenbach1}, the Lamb modes $(u,t,-s,v)$ associated with Dirichlet boundary conditions can be derived. In the symmetric case, they are proportional to
\begin{equation}
\left(\begin{array}{l}
q\cos(qh)\cos(pz)-q\cos(qz)\cos(ph) \\
\mu\bigl(-qp\cos(ph)+k^2\sin(qh)\bigr)\sin(pz)
+\mu\bigl(q^2\cos(ph)-k^2\sin(ph)\bigr)\sin(qz) \\
-ik\bigl((\lambda+2\mu)q\cos(ph)-\lambda p\sin(qh)\bigr)\cos(pz)
+ik\bigl((\lambda+2\mu)q\cos(ph)-\lambda q\sin(ph)\bigr)\cos(qz) \\
-ik\sin(qh)\sin(pz)+ik\sin(qz)\sin(ph)
\end{array}\right),
\end{equation}
whereas in the antisymmetric case they are proportional to
\begin{equation}
\left(\begin{array}{l}
-q\sin(qh)\sin(pz)+q\sin(qz)\sin(ph) \\
\mu\bigl(-qp\sin(ph)+k^2\cos(qh)\bigr)\cos(pz)
+\mu\bigl(q^2\sin(ph)-k^2\cos(ph)\bigr)\cos(qz) \\
ik\bigl((\lambda+2\mu)q\sin(ph)-\lambda p\cos(qh)\bigr)\sin(pz)
-ik\bigl((\lambda+2\mu)q\sin(ph)-\lambda q\cos(ph)\bigr)\sin(qz) \\
-ik\cos(qh)\cos(pz)+ik\cos(qz)\cos(ph)
\end{array}\right).
\end{equation}

\bibliographystyle{plain}
\bibliography{biblio}

\end{document}